\documentclass[11pt]{amsart}
\usepackage{longtable}
\usepackage{array}
\usepackage{multicol}
\usepackage{graphicx}
\usepackage[colorlinks=true,citecolor=black,linkcolor=black,urlcolor=blue]{hyperref}

\makeatletter
 \def\@textbottom{\vskip \z@ \@plus 1pt}
 \let\@texttop\relax
\makeatother
\usepackage{amsmath, amssymb, amsbsy, amsfonts, amsthm, latexsym, amsopn, amstext, amsxtra, euscript, amscd, color, mathrsfs}
\usepackage[normalem]{ulem}
\usepackage{soul}

\usepackage{cite}

\makeatletter

\@namedef{subjclassname@2020}{%
  \textup{2020} Mathematics Subject Classification}
\makeatother

\PassOptionsToPackage{hyphens}{url}\usepackage{hyperref}
 
 \usepackage[capbesideposition=outside,capbesidesep=quad]{floatrow}

\restylefloat{table}
\restylefloat{table}
         
\usepackage{multirow,caption}
            
\usepackage{amscd}
\usepackage{color,enumerate}

\newcommand{\RNum}[1]{\lowercase\expandafter{\romannumeral #1\relax}}

\usepackage[colorinlistoftodos,prependcaption,textsize=tiny]{todonotes}

\theoremstyle{plain}
\newtheorem{thm}{Theorem}[section]
\newtheorem{lem}[thm]{Lemma}

\newtheorem{prop}[thm]{Proposition}

\newtheorem{rmk}[thm]{Remark}

\newtheorem{thm-con}[thm]{Theorem-Conjecture}
\numberwithin{equation}{section}

\theoremstyle{definition}

\newcommand{\F}{\mathbb F}

\def\Tr{{\rm Tr}}

\begin{document}
\title[Permutation polynomials from the trace functions]{Permutation polynomials from the trace functions}
\author[S. U. Hasan]{Sartaj Ul Hasan}
 \address{Department of Mathematics, Indian Institute of Technology Jammu, Jammu 181221, India}
  \email{sartaj.hasan@iitjammu.ac.in}
  
  \author[R. Kaur]{Ramandeep Kaur}
  \address{Department of Mathematics, Indian Institute of Technology Jammu, Jammu 181221, India}
  \email{2022rma0027@iitjammu.ac.in}
  \author[H. Kumar]{Hridesh Kumar}
\address{Department of Mathematics, Indian Institute of Technology Jammu, Jammu 181221, India}
\email{2021rma2022@iitjammu.ac.in}

\thanks{ The second and third named authors are supported by the Prime Minister’s Research Fellowship, under PMRF IDs 3003658 and 3002900, respectively, at IIT Jammu.}
\begin{abstract}

We study necessary and sufficient conditions on $\gamma$ for several classes of polynomials of the form $X+\gamma \operatorname{Tr}_{q}^{q^n}(h(X))$ to be permutation polynomials over finite field $\mathbb{F}_{q^n}$, where $q$ is a prime power, $n$ is a positive integer, and $\operatorname{Tr}_{q}^{q^n}(\cdot)$ denotes the relative trace function from $\mathbb{F}_{q^n}$ to $\mathbb{F}_q$. In addition, we completely characterize the permutation polynomials of the form 
\(
X+\gamma\operatorname{Tr}_{q}^{q^n}(h(X))
\)
over $\mathbb{F}_{q^n}$, with their compositional inverses, where
\(
h(X)=c_1X+c_2X^2+X^2\operatorname{Tr}_{q}^{q^n}(X),
\)
$c_1,c_2\in\mathbb{F}_q.$

\end{abstract}
\maketitle
\section{Introduction}
For a prime power $q=p^m$, where $m$ is a positive integer, let $\F_q$ denote the finite field with $q$ elements, $\F_q^*$  the multiplicative cyclic group of all non-zero elements of $\F_q$, and $\F_q[X]$ the ring of polynomials over $\F_q$ in the indeterminate $X$. It is well-known that every mapping from $\F_{q}$ into itself can be uniquely represented by a polynomial in $\F_{q}[X]$ of degree less than $q$. Accordingly, we may use the terms function and polynomial interchangeably. A polynomial $f(X)\in \F_q[X]$ is a permutation polynomial (PP) if the induced map $c \mapsto f(c)$ is a bijection from $\F_q$ to itself. The study of permutation polynomials was started by Hermite~\cite{CC} over finite fields of prime order. Later, Dickson~\cite{LE} investigated PPs over arbitrary finite fields. For any PP $f(X)\in \F_q[X]$, there exists a unique  polynomial $f^{-1}(X)\in \F_q[X]$ modulo $X^q-X$ such that $f^{-1}(f(X))\equiv f(f^{-1}(X))\equiv X \pmod{X^q-X}$. The polynomial  $f^{-1}$ is known as the compositional inverse of $f$ over $\F_q$. In 1991, Mullen~\cite{M} proposed the problem of computing the coefficients of the compositional inverses of permutation polynomials. Permutation polynomials and their compositional inverses have a wide range of applications in cryptography\cite{RSA,Schwenk_Cr_98}, coding theory \cite{DH,Chapuy_C_07}, and combinatorial design theory ~\cite{DY}.  In particular, in a block cipher with a Substitution-Permutation Network (SPN) structure,  a permutation is often used as an S-box to build the confusion layer during the encryption process, and the compositional inverse is needed while decrypting the cipher. Due to numerous applications in several areas, an extensive research has been devoted to examining various forms of permutation polynomials over finite fields. For more details on permutation polynomials and their compositional inverses see \cite{CK, JLQ, SRK, Hou_PP_15, W} and references therein.

In 2008, Charpin and Kyureghyan \cite{CK} studied permutation polynomials of the form $$G(X)+\gamma \Tr_{2}^{2^m}(H(X))$$ over $\F_{2^m}$ using the linear structure of Boolean functions, where $\Tr_{2}^{2^m}(X)$ is a relative trace. By considering $G(X)$ is a PP or a linearized polynomial, authors derived six classes of such permutation polynomials over $\F_{2^m}$ and later extended this work to finite fields of odd characterstic \cite{CK1}. Kyureghyan and Zieve \cite{KZ}, in 2016, investigated all permutation polynomials of the type $X+\gamma \Tr_{q}^{q^n}(X^k)$, where $k$ is a positive integer, $\gamma \in \F_{q^n}^*$, $q$ odd, $n>1$ and $q^n<5000$ and generalized them into nine infinite classes except five examples. Later, Ma and Ge \cite{MG} generalized two of these examples to a new infinite class. Motivated by their work, Li, Qu, Chen and Li \cite{LQCL} studied polynomials of the form $cX+\Tr_{q}^{q^n}(X^k)$, where $q=2^m$, $c\in \F_{q^n}^*$, $m>1$ and $mn<14$. In 2019, Zha, Hu and Zhang \cite{ZHZ} constructed several classes of permutation polynomials of the form $X+\gamma \Tr_{q}^{q^n}(h(X))$ over $\F_{q^n}$, which explained the remaining examples of Kyureghyan and Zieve \cite{KZ} and one of Li et al. ~\cite{LQCL}. Then authors \cite {ZHZ} investigated permutation polynomials of such forms for $n \in \{2,3,4\}.$

Recently, Jiang, Li, Qu\cite{ JLQ} and Jiang, Yuan, Li, Qu \cite{JYLQ} constructed permutation polynomials of the form $X+\gamma \Tr_{q}^{q^n}(h(X))$ for $n=2,3$, where $q=2^m$.  Then Pang, Wu and Yuan \cite {PWY} studied permutation polynomials of the form $L(X)+\gamma \Tr_{q}^{q^3}(h(X))$ over $\F_{q^3}$ with even characteristics. Moreover, they gave the necessary and sufficient conditions on $\gamma \in \F_q$ for the permutation polynomials given by Jiang, Li and Qu \cite{JLQ}.  Recently, Singh, Kumar and Prakash \cite{SKP} computed the compositional inverses of the permutation polynomials proposed in \cite{JYLQ}.

Determining necessary and sufficient conditions on $\gamma$ for polynomials of the form
\[
X+\gamma \operatorname{Tr}_{q}^{q^n}(h(X))
\]
to be permutation polynomials over $\mathbb{F}_{q^n}$ is a challenging problem. Although numerous permutation polynomials of this form have been constructed, only a limited number of classes have been studied for which necessary and sufficient conditions on $\gamma$ are known. Motivated by this problem, we investigate several classes of polynomials of the form
\(
X+\gamma \operatorname{Tr}_{q}^{q^n}(h(X))
\)
and determine necessary and sufficient conditions on $\gamma$ for them to be permutation polynomials over $\mathbb{F}_{q^n}$ for some $n$.

 
Permutation polynomials and their compositional inverses have attracted considerable attention due to their numerous applications. Nevertheless, determining an explicit expression for the compositional inverse of a permutation polynomial is, in general, a hard problem. In this paper, we completely characterize the permutation polynomials of the form
\(
X+\gamma\operatorname{Tr}_{q}^{q^n}(h(X))
\)
over $\mathbb{F}_{q^n}$, where
\[
h(X)=c_1X+c_2X^2+X^2\operatorname{Tr}_{q}^{q^n}(X),
\qquad c_1,c_2\in\mathbb{F}_q.
\]
Furthermore, we derive explicit polynomial expressions for their compositional inverses.

The remainder of the paper is organized as follows. Section \ref{S2} introduces some known results. In Section \ref{S3}, we study the permutation polynomials of the form $X+\gamma\operatorname{Tr}_{q}^{q^n}(h(X))$ for $n=2, 3.$ Section \ref{S4} investigates the permutation polynomials and their compositional inverses over $\F_{q^n}.$

\section{Preliminaries}\label{S2}
This section collects the known results which we further study in the subsequent sections.
\begin{lem}\label{L21}\cite[Theorem 10]{ZHZ}
Let $n$ and $m$ be positive integers, $q=p^m$, and
$\gamma\in\F_q^{*}$. The polynomial
\(
f(X)=X+\gamma\Tr_{q}^{q^n}\!\left(X^2-X^{q+1}\right)
\)
permutes $\F_{q^n}$.
\end{lem}
\begin{lem}\label{L22}\cite[Theorem 4.1]{JYLQ}
Let $q=2^m$, where $m$ is a positive integer, and $f(X)=X+ \Tr_{q}^{q^2}(X+X^2+X^{2q-1}) \in \F_{q^2}[X]$. Then $f(X)$ is a permutation polynomial over $\F_{q^2}$. 
\end{lem}
\begin{lem}\label{L23}\cite[Theorem 4.2]{JYLQ}
Let $q = 2^m$ and $f(X)=X+\operatorname{Tr}_{q}^{q^2}\left(X+X^3+X^{q+2}+X^{2q-1}\right)$. Then $f(X)$ is a permutation polynomial of $\mathbb{F}_{q^2}$.
\end{lem}

    \begin{lem}\label{L24}\cite[Theorem 3.5]{LQCL}
Let $q = 2^{m}$ and
$
f(X) = \gamma X + \Tr_{q^{2n+1}/q}\!\left(X^{\frac{q^{2} + q}{2}}\right)$,
where $\gamma \in \mathbb{F}_{q} \setminus \{0,1\}$ and $n, m > 0$ are integers. 
Then $f(X)$ is a permutation polynomial over $\mathbb{F}_{q^{2n+1}}$.
\end{lem}

\begin{lem}\label{L25}\cite[Corollary 1]{ZHZ}
    Let $n,m$ be positive integers and $q=2^m$ and $\gamma\in\mathbb{F}_q^*$. The polynomial
\(
f(X)
=
X+\gamma\,\operatorname{Tr}_{q}^{q^n}\!\left(X^3-X^{q+2}\right)
\)
permutes $\mathbb{F}_{q^n}$.

\end{lem}
\begin{lem}\label{L26}\cite[Theorem 3.1]{JYLQ}
Let $q = 2^m$ and $f(X) = X + \gamma \operatorname{Tr}_{q}^{q^2}(X^3 + X^{q+2})$.
Then $f(X)$ is a permutation polynomial of $\mathbb{F}_{q^2}$ if and only if $\gamma \in \mathbb{F}_q$.
\end{lem}
\begin{lem}\label{L27}\cite[Theorem 3.3]{JYLQ}
Let $q = 2^m$ and $f(X) = X + \gamma \operatorname{Tr}_{q}^{q^2}(X+X^2+X^3 + X^{q+2})$.
Then $f(X)$ is a permutation polynomial of $\mathbb{F}_{q^2}$ if and only if $\gamma \in \mathbb{F}_q$.
\end{lem}
\begin{lem}\label{L28}\cite[Theorem 3.5]{JYLQ}
Let $q = 2^m$ and $f(X) = X + \gamma \operatorname{Tr}_{q}^{q^2}(X+X^3 + X^{q+2})$.
Then $f(X)$ is a permutation polynomial of $\mathbb{F}_{q^2}$ if and only if 
\begin{enumerate}[(a)]
\item $\gamma \in \mathbb{F}_q$ if $m$ is even,
\item $\Tr_{q}^{q^2}(\gamma+\gamma^2)=0$ if $m$ is odd.
\end{enumerate}
\end{lem}
\begin{lem}\label{L29}\cite[Theorem 3.6]{JYLQ}
Let $q = 2^m$ and $f(X) = X + \gamma \operatorname{Tr}_{q}^{q^2}(X^2+X^3 + X^{q+2})$.
Then $f(X)$ is a permutation polynomial of $\mathbb{F}_{q^2}$ if and only if 
\begin{enumerate}[(a)]
\item $\gamma \in \mathbb{F}_q$ if $m$ is even,
\item $\Tr_{q}^{q^2}(\gamma+\gamma^2)=0$ if $m$ is odd.
\end{enumerate}
\end{lem}

\section{Permutation polynomials over $\F_{q^2}$ and $\F_{q^3}$}\label{S3}
In the following theorem, we provide the necessary and sufficient condition on $\gamma$ for Lemma \ref{L21} in the case where $q=2^m$ and $n=3.$ 

   \begin{thm}\label{T31}
    Let $q=2^m$, where $m$ is a positive integer. Let $f(X)=X+\gamma \Tr_{q}^{q^3}(X^2+X^{q+1}) \in \F_{q^3}[X].$ Then $f(X)$ is a polynomial over $\F_{q^3}$ if and only if $\gamma \in \F_q.$ 
\end{thm}
\begin{proof}
    It is immediate that $f(X)$ is a permutation polynomial over $\F_{q^3}$ for $\gamma=0.$ We restrict to the case $\gamma \neq 0$ and aim to prove that $f(X)$ is a permutation polynomial if and only if $\gamma \in \F_q.$ For this, we will show that  for any $a \in \F_{q^3}$, the equation $$f(X)=a$$  has a unique solution in $\F_{q^3}$ if and only if $\gamma \in \F_q.$ Now, the equation $f(X)=a$ implies that
    \begin{equation}\label{e36}
    X+\gamma\Tr_{q}^{q^3}(X^2+X^{q+1})=a.    
    \end{equation}
    Let $\frac{X+a}{\gamma}=u:=\Tr_{q}^{q^3}(X^2+X^{q+1}) \in \F_q$. We substitute $X=u\gamma+a$ into Equation \eqref{e36}. This yields 
    \begin{equation*}
        u+\Tr_{q}^{q^3}(u^2\gamma^2+a^2+u^2\gamma^{q+1}+u\gamma^qa+u\gamma a^q+a^{q+1})=0,
    \end{equation*}
    which is equivalent to 
    \begin{equation}\label{e37}
        u^2\Tr_{q}^{q^3}(\gamma^2+\gamma^{q+1})+u(\Tr_{q}^{q^3}(\gamma^qa+\gamma a^q)+1)+\Tr_{q}^{q^3}(a^2+a^{q+1})=0.
    \end{equation}
   Consequently, $f(X)$ permutes $\F_{q^3}$ if and only if Equation \eqref{e37} has a unique solution in $\F_q$ for any $a \in \F_{q^3}$. 

If $\gamma \in \F_q$, then $u=\Tr_{q}^{q^3}(a^2+a^{q+1})$ is a unique solution of Equation \eqref{e37}. Hence, $f(X)$ is a permutation polynomial for $\gamma \in \F_q.$ We now assume that $\gamma \in \F_{q^3} \setminus \F_q.$

\textbf{Case 1.} If $\Tr_{q}^{q^3}(\gamma^2) \neq \Tr_{q}^{q^3}(\gamma^{q+1})$. Then for $a \in \F_q$, Equation \eqref{e37} has two distinct solutions, namely, $$u=0 \text{ and }  u=\frac{1}{\Tr_{q}^{q^3}(\gamma^2+\gamma^{q+1})}.$$ Therefore, $f(X)$ is not a permutation polynomial over $\F_{q^3}$.

\textbf{Case 2.} Suppose that $\Tr_{q}^{q^3}(\gamma^2)=\Tr_{q}^{q^3}(\gamma^{q+1})$.  The equality  $$\Tr_{q}^{q^3}(\gamma^2)=\Tr_{q}^{q^3}(\gamma^{q+1})$$ implies that 
\begin{equation*}
    \gamma^2(1+\gamma^{2q-2}+\gamma^{2q^2-2}+\gamma^{q-1}+\gamma^{q^2+q-2}+\gamma^{q^2-1})=0.
\end{equation*}
Since $\gamma \neq 0$, this reduces to 
$$1+\gamma^{2q-2}+\gamma^{2q^2-2}+\gamma^{q-1}+\gamma^{q^2+q-2}+\gamma^{q^2-1}=0.$$
 Let $y=\gamma^{q-1}$, then above equation becomes 
 \begin{equation*}
     1+y^2+y^{2(q+1)}+y+y^{q+2}+y^{q+1}=0.
 \end{equation*}
 We now use $z=y^{q+1}$ in the above equation to obtain
 \begin{equation}\label{e38}
     z^2+(y+1)z+y^2+y+1=0.
 \end{equation}
Equation \eqref{e38} has a solution in $\F_{q^3}$ if and only if $\Tr_{2}^{2^{3m}}\left(\frac{y^2+y+1}{y^2+1}\right)=0.$ Observe that 
$$\Tr_{2}^{2^{3m}}\left(\frac{y^2+y+1}{y^2+1}\right)=\Tr_{2}^{2^{3m}}\left(\frac{y}{y^2+1}+1\right).$$
When $m$ is odd, we have
$$\Tr_{2}^{2^{3m}}\left(\frac{y}{y^2+1}+1\right)=\Tr_{2}^{2^{3m}}\left(\frac{y}{y^2+1}\right)+1.$$
Moreover, $$\Tr_{2}^{2^{3m}}\left(\frac{y}{y^2+1}\right)+1=\Tr_{2}^{2^{3m}}\left(\frac{1}{y+1}+\frac{1}{y^2+1}\right)+1=1.$$
Therefore, if $m$ is odd, Equation \eqref{e38} has no solution in $\F_q$. Thus, there does not exist any $\gamma \in \F_{q^3} \setminus \F_q$ such that $\Tr_{q}^{q^3}(\gamma^2)=\Tr_{q}^{q^3}(\gamma^{q+1}).$ Consequently, we assume that $m$ is an even positive integer. Since $\Tr_{q}^{q^3}(\gamma^2)=\Tr_{q}^{q^3}(\gamma^{q+1})$, therefore, Equation \eqref{e37} reduces to
\begin{equation}\label{e39}
    u(\Tr_{q}^{q^3}(\gamma^qa+\gamma a^q)+1)+\Tr_{q}^{q^3}(a^2+a^{q+1})=0.
\end{equation}
Let $a=\frac{1}{\gamma^q+\gamma^{q^2}}.$
Substituting this value into $\Tr_{q}^{q^3}(a^2+a^{q+1})$, we obtain 
\begin{equation*}
\begin{split}
\Tr_{q}^{q^3}\left(a^2+a^{q+1}\right)=&\frac{1}{\gamma^{2q}+\gamma^{2q^2}}+\frac{1}{\gamma^{2q^2}+\gamma^{2}}+\frac{1}{\gamma^{2}+\gamma^{2q}}+\frac{1}{(\gamma^{q}+\gamma^{q^2})(\gamma+\gamma^{q^2})}+\frac{1}{(\gamma+\gamma^{q})(\gamma+\gamma^{q^2})}+\\&\frac{1}{(\gamma^{q}+\gamma^{q^2})(\gamma+\gamma^{q})},
\end{split}
\end{equation*}
which is equivalent to
\begin{equation}\label{e310}
\begin{split}
\Tr_{q}^{q^3}\left(a^2+a^{q+1}\right)=&
\frac{1}{\gamma^{2q}+\gamma^{2q^2}}+\frac{1}{\gamma^{2q^2}+\gamma^{2}}+\frac{1}{\gamma^{2}+\gamma^{2q}}+\frac{1}{\gamma^{q^2+q}+\gamma^{2q^2}+\gamma^{q+1}+\gamma^{q^2+1}}+\\&\frac{1}{\gamma^{q^2+1}+\gamma^{2}+\gamma^{q^2+q}+\gamma^{q+1}}+\frac{1}{\gamma^{q+1}+\gamma^{q^2+1}+\gamma^{2q}+\gamma^{q^2+q}}.
\end{split}
\end{equation}
As $\Tr_{q}^{q^3}(\gamma^2)=\Tr_{q}^{q^3}(\gamma^{q+1})$, it follows that 
\begin{equation}\label{e311}
    \gamma^2+\gamma^{2q}+\gamma^{2q^2}=\gamma^{q+1}+\gamma^{q^2+q}+\gamma^{q^2+1}.
\end{equation}
By using Equation \eqref{e311} into Equation \eqref{e310}, we conclude that
$$\Tr_{q}^{q^3}(a^2+a^{q+1})=0.$$
Moreover, for $a=\frac{1}{\gamma^q+\gamma^{q^2}}$, we have 
$$\Tr_{q}^{q^3}(\gamma^qa+\gamma a^q)+1=\Tr_{q}^{q^3}\left(a(\gamma^q+\gamma^{q^2})\right)+1=0.$$  Therefore, Equation \eqref{e39} admits $q$ distinct solutions in $\F_q$ for $a=\frac{1}{\gamma^q+\gamma^{q^2}}.$ Consequently, $$f(X)=\frac{1}{\gamma^q+\gamma^{q^2}}$$ has at least $q$ solutions in $\F_{q^3}$ and thus in this case also, $f(X)$ is not a permutation polynomial over $\F_{q^3}.$

\end{proof}

In Lemma \ref{L22}, the authors established a class of permutation polynomials over $\F_{q^2}$ by taking $\gamma=1$. We now  determine precisely which elements $\gamma \in \F_q$ preserve the conclusion of Lemma \ref{L22}, when $m$ is odd.

\begin{thm}\label{T32}
    Let $q=2^m$ and $f(X)=X+\gamma \Tr_{q}^{q^2}\left(X+X^2+X^{2q-1}\right) \in \F_{q^2}[X]$ be a polynomial, where $m$ is an odd positive integer and $\gamma \in \F_q$. Then $f(X)$ is a permutation polynomial over $\F_{q^2}$ if and only if $\gamma \in \{0,1\}$.
    \end{thm}
\begin{proof} Clearly for $\gamma=0$, $f(X)$ is a permutation polynomial. Therefore we may assume that $\gamma \neq 0$. For any $a \in \F_{q^2}$ consider the equation $f(X)=a$, that is,
\begin{equation}\label{e316}
X+\gamma \Tr_{q}^{q^2}\left(X+X^2+X^{2q-1}\right)=a.
\end{equation} 
 Let $u:=\Tr_{q}^{q^2}\left(X+X^2+X^{2q-1}\right) \in \F_q$, which gives us $X=a+\gamma u$. 
For $a=0$, put $X=\gamma u$ in Equation \eqref{e316} to get 
\[
\gamma u+\gamma \Tr_{q}^{q^2}\left(\gamma u +\gamma^2 u^2+\gamma^{2q-1}u^{2q-1}\right)=0, 
\]
which yields $u=0$, that is, $X=0$. Therefore, for $a=0$, $X=0$ is the only solution of Equation \eqref{e316}. We now assume that $a\neq 0$. Clearly, $X=0$ cannot be a solution of \eqref{e316} in this case. We substitute $X=a+\gamma u \neq 0$ in Equation \eqref{e316} and obtain the following equation
\[
a+\gamma u+\gamma \Tr_{q}^{q^2}\left(a+\gamma u+a^2+\gamma^2 u^2+\frac{(a+\gamma u)^{2q}}{a+\gamma u}\right)=a.
\]
Since $\gamma \neq 0$, we get
\[
u+\Tr_{q}^{q^2}\left(a+\gamma u+a^2+\gamma^2 u^2+\frac{(a+\gamma u)^{2q}}{a+\gamma u}\right)=0.
\]
Now using the fact that  $\gamma \in \F_q^{*}$ and properties of trace function, the above equation can be reduced to
\[
u+a+a^{q}+a^2+a^{2q}+\frac{a^{2q}+\gamma^2u^2}{a+\gamma u}+\frac{a^2+\gamma^2u^2}{(a+\gamma u)^q}=0,
\]
which gives 
\[
(a+\gamma u)(a+\gamma u)^q(u+a+a^q+a^2+a^{2q})+(a^{2q}+\gamma^2u^2)(a+\gamma u)^q+(a^2+\gamma^2u^2)(a+\gamma u)=0.
\]
After some simplifications, we obtain
\begin{equation*}
\begin{split}
     \gamma^2 u^3+(\gamma(a+a^q)+\gamma^2(a+a^q)^2)u^2+\left(a^{q+1}+\gamma(a^3+a^{3q}+a^{2q+1}+a^{q+2})\right)u+a^{q+2}+a^{2q+1}+a^{q+3}+\\
     a^{3q+1}+a^{3q}+a^3=0, 
    \end{split}
\end{equation*}
which is equivalent to 
\begin{equation}\label{e317}
\gamma^2 u^3+\gamma(a+a^q)(1+\gamma(a+a^q))u^2+\left(a^{q+1}+\gamma(a+a^{q})^3\right)u+(a+a^q)^2(a+a^q+a^{q+1})=0. 
\end{equation}
Note that the equation 
\(
X^2+\frac{X}{\gamma}+\frac{1}{\gamma^2}
\)
has no solution in $\F_q$ since $\Tr_{2}^{q}\left(\dfrac{\frac{1}{\gamma^2}}{\frac{1}{\gamma^2}}\right)=\Tr_{2}^{q}(1)=1$, and $\Tr_{2}^{q}(1)=1$ as $q=2^m$ and $m$ is odd. Now choose $a \in \F_{q^2}\setminus \F_q$ in Equation \eqref{e317} such that 
\begin{equation}\label{e318}
\begin{split}
a^2+\frac{a}{\gamma}+\frac{1}{\gamma^2}=0.
\end{split}
\end{equation}
Raise $q$ power to Equation \eqref{e318} and add this to  \eqref{e318} to obtain
\[
(a+a^q)^{2}+\frac{a+a^q}{\gamma}=0, 
\]
which gives that $\left(a+a^q\right)\left(a+a^q+\frac{1}{\gamma}\right)=0$. Thus, $a+a^q+\frac{1}{\gamma}=0$, i.e., $\gamma(a+a^q)+1=0$, as $a \not \in \F_q$. Moreover, $a+a^q+\frac{1}{\gamma}=0$ implies that $a^{q+1}=a^{2}+\frac{a}{\gamma}$. Thus, in Equation \eqref{e317}, the coefficient of $u^2$ is 
\[
\gamma(a+a^q)(1+\gamma(a+a^q))=0
\]
and the coefficient of $u$ is 
\[
a^{q+1}+\gamma(a+a^{q})^3=a^{2}+\frac{a}{\gamma}+\frac{1}{\gamma^2}=0.
\]
Finally, Equation \eqref{e317} reduces to 
\begin{equation}\label{e319}
\gamma^2 u^3+(a+a^q)^2\left(a+a^q+a^{q+1}\right)=\gamma^2 u^3+\frac{1}{\gamma^2}\left(\frac{1}{\gamma}+a^2+\frac{a}{\gamma}\right)=0.
\end{equation}
For $\gamma \not \in \{0,1\}$,  we have $\frac{1}{\gamma}+a^2+\frac{a}{\gamma} \neq 0$, otherwise, $\frac{1}{\gamma}+\frac{1}{\gamma^2}=0$, which gives $\gamma=0$ or $1$. Therefore, for $a$ satisfying $
X^2+\frac{X}{\gamma}+\frac{1}{\gamma^2}$, Equation \eqref{e319} has three distinct solutions and $f(X)$ is not a permutation polynomial for $\gamma \not \in \{0,1\}$. Consequently, $f(X)$ is a permutation polynomial if and only if $\gamma \in \{0,1\}$.
\end{proof}

The construction presented in Lemma \ref{L23} is obtained under the assumption $\gamma=1$. We now remove this restriction and determine $\gamma \in \F_q$ for which the conclusion of Lemma \ref{L23} holds, when $m$ is odd. This characterization is given in the following theorem.

\begin{thm}\label{T33}
Let $q = 2^m$ and $f(X)=X+\gamma \Tr_{q}^{q^2}\left(X+X^3+X^{q+2}+X^{2q-1}\right)$, where $m$ is an odd positive integer and $\gamma \in \F_q$. Then $f(X)$ is a permutation polynomial over $\mathbb{F}_{q^2}$ if and only if $\gamma \in \{0,1\}$.
\end{thm}
\begin{proof}
It is evident that \(f(X)\) is a permutation polynomial when \(\gamma=0\). Hence, throughout the remainder of the proof, we assume that \(\gamma\neq 0\). By Lemma~\ref{L23}, \(f(X)\) is a permutation polynomial when \(\gamma=1\). Therefore, it suffices to show that \(f(X)\) is not a permutation polynomial for \(\gamma\notin\{0,1\}\).

To this end, consider the equation
\begin{equation}\label{e320}
X+\gamma \Tr_{q}^{q^2}\left(X+X^3+X^{q+2}+X^{2q-1}\right)=a,
\end{equation}
where \(a\in\F_{q^2}^{*}\). We will show that, for every \(\gamma\notin\{0,1\}\), there exists an element \(a\in\F_{q^2}^{*}\) such that Equation \eqref{e320} admits more than one solution in \(\F_{q^2}\). This implies that \(f(X)\) is not injective and hence cannot be a permutation polynomial.

Assume that $u:=\Tr_{q}^{q^2}\left(X+X^3+X^{q+2}+X^{2q-1}\right) \in \F_q$, and $X=a+\gamma u$. Since for $a \neq 0$, $X=0$ cannot be a solution of Equation \eqref{e320}. Therefore, in this case we have $X=a+\gamma u \neq 0$ and substitute $X=a+\gamma u$ into Equation \eqref{e320} to obtain
\[
u+\Tr_{q}^{q^2}\left(\gamma u+a+(\gamma^2 u^2+a^2)(\gamma u+a)+(\gamma u+a^q)(\gamma^2 u^2+a^2)+\frac{\gamma^2u^2+a^{2q}}{\gamma u+a}\right)=0
\]
as $\gamma \neq 0$.
The above equation gives 
\[
u+\Tr_{q}^{q^2}(a+a^3+a^{q+2})+\frac{\gamma^2 u^2+a^{2q}}{\gamma u+a}+\frac{\gamma^2 u^2+a^2}{\gamma u +a^q}=0.
\]
After doing some simplifications, we get 
\begin{equation*}
\begin{split}
&u^3\gamma^2+u^{2} \gamma (a+a^q)(1+\gamma (a^2+a^{2q}+1)+\gamma)+u\left(a^{q+1}+\gamma (a+a^q)^2(a^2+a^{2q}+1)+\gamma a^{2q}+\gamma a^2\right)\\&+(a+a^q)^3(a^{q+1}+1)=0.
\end{split}
\end{equation*}
This implies that
\begin{equation}\label{e321}
\begin{split}
u^3\gamma^2+u^{2} \gamma (a+a^q)(1+\gamma (a+a^{q})^2)+u\left(a^{q+1}+\gamma (a+a^q)^4\right)+(a+a^q)^3(a^{q+1}+1)=0.
\end{split}
\end{equation}
Since $\Tr_{2}^{q}\left(\dfrac{\frac{1}{\gamma}}{\frac{1}{\gamma}}\right)=\Tr_{2}^{q}(1)\neq 0$, the equation $X^2+\frac{X}{\sqrt{\gamma}}+\frac{1}{\gamma}=0$ has no solutions in $\F_q$. Thus, we may choose $a \in \F_{q^2}\setminus \F_q$ such that
\begin{equation}\label{E31}
    a^2+\frac{a}{\sqrt{\gamma}}+\frac{1}{\gamma}=0.
\end{equation}

Raise $q$ power to Equation \eqref{E31} and add this to  \eqref{E31} to get
\[
(a+a^q)^{2}+\frac{a+a^q}{\sqrt{\gamma}}=0, 
\]
which gives that $\left(a+a^q\right)\left(a+a^q+\frac{1}{\sqrt{\gamma}}\right)=0$. Thus, $a+a^q+\frac{1}{\sqrt{\gamma}}=0$, i.e., $\sqrt{\gamma}(a+a^q)+1=0$, as $a \not \in \F_q$.
Moreover, 
\[
a+a^{q}=\frac{1}{\sqrt{\gamma}}
\] 
implies $a^{q+1}=a^2+\frac{a}{\sqrt{\gamma}}$. Substituting these expressions into equation \eqref{e321}, we find that both the coefficient of \(u^2\) and the coefficient of \(u\) vanish. Thus, Equation \eqref{e321} becomes 
\[
u^3 \gamma^2+\frac{1}{\gamma^{\frac{3}{2}}}\left(a^2+\frac{a}{\sqrt{\gamma}}+1\right)=0.
\]
The constant term of the above equation is also nonzero. Indeed, if it was zero, then
\(
a^2+\frac{a}{\sqrt{\gamma}}+1=0,
\)
would imply that $\gamma=1$, contradicting the assumption that $\gamma\notin\{0,1\}$. Hence, the equation has three distinct solutions. Consequently, there exists an element $a\in\F_{q^2}^{*}$ with three distinct pre-images under $f$, implying that $f(X)$ is not a permutation polynomial for $\gamma\notin\{0,1\}$. This completes the proof.

\end{proof}
\section{Permutation polynomials over $\F_{q^n}$}\label{S4}

The construction in Lemma~\ref{L24} is established under the assumption that 
$\gamma \in \mathbb{F}_{q} \setminus \{0,1\}$. The following theorem characterizes precisely those values of $\gamma$ for which the conclusion of Lemma \ref{L24} remains true. Furthermore, we provide the compositional inverse.

 \begin{thm}\label{T34}
Let $q=2^m$ and $f(X)=X+\gamma \Tr_{q}^{q^n}\left(X^{\frac{q+q^2}{2}}\right)$ be a polynomial over $\F_{q^n}$, where $m$ is a positive integer and $n$ is an odd positive integer. Then $f(X)$ is a permutation polynomial if and only if $\gamma \in \F_{q} \setminus\{1\}$. Moreover, $f^{-1}(X)=X+\frac{\gamma}{\gamma+1}\Tr_{q}^{q^n}\left(X^{\frac{q+q^2}{2}}\right)$.
\end{thm}
\begin{proof}
    Clearly, $f(X)$ is a permutation polynomial over $\F_{q^n}$ when $\gamma=0$, and therefore, we consider $\gamma \neq 0.$ We will show that $f(X)=a$ has a unique solution in $\F_{q^n}$ if and only if $\gamma \in \F_{q} \setminus\{1\}$. For $f(X)=a$, we have 
    \begin{equation}\label{e31}
        X+\gamma \Tr_{q}^{q^n}\left(X^{\frac{q+q^2}{2}}\right)=a.
    \end{equation}
  Let $\frac{X+a}{\gamma}=\Tr_{q}^{q^n}\left(X^{\frac{q+q^2}{2}}\right):=u \in \F_q$. We put $X=u\gamma+a$ in Equation \eqref{e31} to get 
    \begin{equation*}
        u+\Tr_{q}^{q^n}\left(u\gamma^{\frac{q^2+q}{2}}+u^{\frac{1}{2}}\gamma^{\frac{q^2}{2}}a^{\frac{q}{2}}+u^{\frac{1}{2}}\gamma^{\frac{q}{2}}a^{\frac{q^2}{2}}+a^{\frac{q^2+q}{2}}\right)=0,
    \end{equation*}
    which is equivalent to 
    \begin{equation*}
        u\left(\Tr_{q}^{q^n}\left(\gamma^{\frac{q^2+q}{2}}\right)+1\right)+u^{\frac{1}{2}}\Tr_{q}^{q^n}\left(\gamma^{\frac{q^2}{2}}a^{\frac{q}{2}}+\gamma^{\frac{q}{2}}a^{\frac{q^2}{2}}\right)+\Tr_{q}^{q^n}\left(a^{\frac{q^2+q}{2}}\right)=0
    \end{equation*}
    or
    \begin{equation}\label{e32}
        u^2\left(\Tr_{q}^{q^n}(\gamma^{q^2+q})+1\right)+u\Tr_{q}^{q^n}\left(\gamma^{q^2}a^{q}+\gamma^qa^{q^2}\right)+\Tr_{q}^{q^n}\left(a^{q^2+q}\right)=0.
    \end{equation}
   $f(X)$ is a permutation polynomial over $\F_{q^n}$ if and only if Equation \eqref{e32} has a unique solution in $\F_q$ for each $a \in \F_{q^n}$. We prove this by the following cases.

    \textbf{Case 1.} Let $\gamma=1$. For any $a \in \F_{q^n}$ satisfying $$\Tr_{q}^{q^n}\left(a^{q^2+q}\right)=0,$$ we have more than one solution of Equation \eqref{e32}. Therefore, $f(X)$ is not a permutation polynomial over $\F_{q^n}$ if $\gamma=1.$
    
  \textbf{Case 2.} If $\gamma \in \F_{q}^* \setminus\{1\}$. Then Equation \eqref{e32} implies
  \begin{equation*}
      u^2(\gamma^2+1)+\Tr_{q}^{q^n}\left(a^{q^2+q}\right)=0,
  \end{equation*}
  which yields that 
  $$u=\frac{\Tr_{q}^{q^n}\left(a^{\frac{q^2+q}{2}}\right)}{\gamma+1}$$  
  is a unique solution of Equation \eqref{e32}. Thus, $f(X)=a$ has a unique solution in $\F_{q^n}$ for $\gamma \in \F_{q}^* \setminus\{1\}.$

  \textbf{Case 3.} Suppose that $\gamma \in \F_{q^n} \setminus \F_q$. If $$\Tr_{q}^{q^n}(\gamma^{q^2+q})=1,$$ then Equation \eqref{e32} is trivially satisfied for $a=0$. Consequently, when $a=0$, Equation \eqref{e32} admits $q$ distinct solutions in $\F_q$. This implies that Equation \eqref{e31}, $f(X)=0$, has more than one solutions in $\F_{q^n}$. Hence, $f(X)$ cannot be a permutation polynomial. Now, assume that $$\Tr_{q}^{q^n}(\gamma^{q^2+q})\neq1.$$ For any $\gamma \in \F_{q^n} \setminus \F_q$, there are at most $q^{n-1}$ elements $a \in \F_{q^n}$ satisfying $$\Tr_{q}^{q^n}(\gamma^{q^2}a^{q}+\gamma^qa^{q^2})=0.$$ 
  Thus, for any $\gamma \in \F_{q^n} \setminus \F_q$, there exists $a \in \F_{q^n}$ such that $$\Tr_{q}^{q^n}(\gamma^{q^2}a^{q}+\gamma^qa^{q^2})\neq 0,$$
  and  Equation \eqref{e32} has either two distinct solutions or no solution for such $a$. This again implies that $f(X)$ fails to be a permutation polynomial. Hence, $f(X)$ is a permutation polynomial over $\F_{q^n}$ if and only if $\gamma \in \F_{q} \setminus\{1\}$.

  Next, we show that $g(X)=X+\frac{\gamma}{\gamma+1}\Tr_{q}^{q^n}\left(X^{\frac{q+q^2}{2}}\right)$ is a compositional inverse of $f(X)$. Assume that $g(X)=X+b$, where $b=\frac{\gamma}{\gamma+1}\Tr_{q}^{q^n}\left(X^{\frac{q+q^2}{2}}\right) \in \F_q.$ Now, 

\begin{equation*}
    \begin{split}
        f(g(X))=&X+b+\gamma \Tr_{q}^{q^n}\left((X+b)^{\frac{q+q^2}{2}}\right)
        \\=&X+b+\gamma \Tr_{q}^{q^n}\left(X^{\frac{q+q^2}{2}}+b\right)
        \\=&X+b(\gamma+1)+\gamma \Tr_{q}^{q^n}\left(X^{\frac{q+q^2}{2}}\right)
        \\=&X+\gamma \Tr_{q}^{q^n}\left(X^{\frac{q+q^2}{2}}\right)+\gamma \Tr_{q}^{q^n}\left(X^{\frac{q+q^2}{2}}\right)
        \\=&X.
    \end{split}
\end{equation*}
Hence, $f^{-1}(X)=X+\frac{\gamma}{\gamma+1}\Tr_{q}^{q^n}\left(X^{\frac{q+q^2}{2}}\right)$.
\end{proof}
The following theorem completely characterizes the values of $\gamma$
 for which the conclusion of Lemma \ref{L25} holds and explicitly determines the corresponding compositional inverse.
\begin{thm}\label{T35}
    Let $q=2^m$, where $m$ is a positive integer. Let $f(X)=X+\gamma\Tr_{q}^{q^n}\left(X^3+X^{q+2}\right)$ be a polynomial over $\F_{q^n}$, where $n$ is a positive integer. Then $f(X)$ is a permutation polynomial over $\F_{q^n}$ if and only if $\gamma \in \F_{q}$. Moreover, $f(X)$ is an involution over $\F_{q^n}.$
\end{thm}
\begin{proof}
    $f(X)$ is a permutation polynomial over $\F_{q^n}$ when $\gamma=0$. Henceforth, we assume that $\gamma \neq 0$ and show that, for each $a \in \F_{q^n}$, $f(X)=a$  has a unique solution in $\F_{q^n}$ if and only if $\gamma \in \F_{q}$. For $f(X)=a$, we have
    \begin{equation}\label{e33}
        X+\gamma\Tr_{q}^{q^n}\left(X^3+X^{q+2}\right)=a.
    \end{equation}
  Let $\frac{X+a}{\gamma}=\Tr_{q}^{q^n}(X^3+X^{q+2}) :=u\in \F_q.$ Substituting $X=u\gamma+a$ into Equation \eqref{e33} yields
    \begin{equation*}
        u+\Tr_{q}^{q^n}\left(u^3\gamma^3+u^2\gamma^2a+u\gamma a^2+a^3+u^3\gamma^{q+2}+u\gamma^qa^2+u^2\gamma^2a^q+a^{q+2}\right)=0.
    \end{equation*}
    The above equation is equivalent to 
    \begin{equation}\label{e34}
        u^3\Tr_{q}^{q^n}\left(\gamma^3+\gamma^{q+2}\right)+u^2\Tr_{q}^{q^n}\left(\gamma^2a+\gamma^2a^q\right)+u(\Tr_{q}^{q^n}(\gamma a^2+\gamma^qa^2)+1)+\Tr_{q}^{q^n}(a^3+a^{q+2})=0.
    \end{equation}
   $f(X)$ is a permutation polynomial over $\F_{q^n}$ if and only if Equation \eqref{e34} has a unique solution in $\F_q$ for any $a \in \F_{q^n}$. We now proceed to prove the result by considering the following two cases.
    
    \textbf{Case 1.} If $\gamma \in \F_q$, then $$u=\Tr_{q}^{q^n}(a^3+a^{q+2})$$ is a unique solution of Equation \eqref{e34}. Hence, Equation \eqref{e33}, $f(X)=a$ has a unique solution in $\F_{q^n}$ for any $a \in \F_{q^n}.$

    \textbf{Case 2.} Let $\gamma \in \F_{q^n} \setminus \F_{q}.$ Suppose that 
    $$\Tr_{q}^{q^n}\left(\gamma^3\right)\neq \Tr_{q}^{q^n}\left(\gamma^{q+2}\right).$$ 
    Then, for $a\in \F_q$, Equation \eqref{e34} admits exactly two distinct solutions $$u=0 \text{ and } u=\frac{1}{\Tr_{q}^{q^n}(\gamma^3+\gamma^{q+2})^{\frac{1}{2}}}.$$
    Therefore, if $\Tr_{q}^{q^n}(\gamma^3)\neq \Tr_{q}^{q^n}(\gamma^{q+2})$ and $a\in \F_q$, then Equation \eqref{e33} has more than one solutions in $\F_{q^n}.$ Now assume that 
    $$\Tr_{q}^{q^n}\left(\gamma^3\right)=\Tr_{q}^{q^n}\left(\gamma^{q+2}\right).$$
    In this case, Equation \eqref{e34} reduces to 
    \begin{equation}\label{e35}
        u^2\Tr_{q}^{q^n}(\gamma^2a+\gamma^2a^q)+u\left(\Tr_{q}^{q^n}(\gamma a^2+\gamma^qa^2)+1\right)+\Tr_{q}^{q^n}\left(a^3+a^{q+2}\right)=0.
    \end{equation}
    For fixed $\gamma \in \F_{q^n} \setminus \F_q$, define the set 
    $$S=\{a \in \F_{q^n} \mid \Tr_{q}^{q^n}(\gamma^2 a+\gamma^2a^q)=0 \text{ or } \Tr_{q}^{q^n}(\gamma a^2+\gamma^q a^2)=1\}.$$
    It is clear that $|S|\leq 2q^{n-1}$. For any $a \in \F_{q^n} \setminus S$, the quadratic Equation \eqref{e35} has either two distinct solutions or no solution in $\F_q$. Therefore, in this case, for all $a \in \F_{q^n} \setminus S$, the equation $f(X)=a$ has either two distinct solutions or no solution in $\F_{q^n}.$

    We now show that $f(X)$ is an involution when $\gamma \in \F_{q}.$ Let $u=\Tr_{q}^{q^n}(X^3+X^{q+2})$. Then 
    $$f(f(X))=X+\gamma u+\gamma \Tr_{q}^{q^n}((X+\gamma u)^3+(X+\gamma u)^{q+2}).$$
 A direct computation yields
    $$X+\gamma u+\gamma \Tr_{q}^{q^n}((X+\gamma u)^3+(X+\gamma u)^{q+2})=X+\gamma \Tr_{q}^{q^n}(X\gamma^2u^2+X^q\gamma^2u^2).$$
    Since $\gamma,u \in \F_q$,
    $$X+\gamma \Tr_{q}^{q^n}(X\gamma^2u^2+X^q\gamma^2u^2)=X.$$
     Hence, $f(f(X))=X$ and therefore $f(X)$ is an involution over
$\F_{q^n}$ for all $\gamma \in \F_q$.
This completes the proof.
\end{proof}

\begin{thm}\label{T36}
    Let $q=2^m$, where $m$ is a positive integer. Let $f(X)=X+\gamma \Tr_{q}^{q^n}(c_1X+c_2X^2+X^2\Tr_{q}^{q^n}(X))\in \F_{q^n}[X]$, where $c_1, c_2 \in \F_q$. Then $f(X)$ is a permutation polynomial over $\F_{q^n}$ if and only if one of the following conditions holds
    \begin{itemize}
        \item $\Tr_{q}^{q^n}(\gamma)=0$,
        \item $\Tr_{q}^{q^n}(\gamma)(c_1+{c_2}^2)=1$ and m is odd.
    \end{itemize}
\end{thm}
\begin{proof}
If $\gamma=0$, then it is clear that $f(X)$ is a permutation polynomial over $\F_{q^n}.$ Therefore, we assume that $\gamma \neq 0.$ Consider $f(X)=a$, for $a \in \F_{q^n},$ that is, 
\begin{equation}\label{e312}
    X+\gamma \Tr_{q}^{q^n}\left(c_1X+c_2X^2+X^2\Tr_{q}^{q^n}(X)\right)=a.
\end{equation}
    Let $\frac{X+a}{\gamma}=\Tr_{q}^{q^n}(c_1X+c_2X^2+X^2\Tr_{q}^{q^n}(X)):=u \in \F_q.$ We put $X=a+\gamma u$ in Equation \eqref{e312} to obtain 
    \begin{equation}\label{e313}
    \begin{split}
          u^3\Tr_{q}^{q^n}(\gamma)^3+u^2\left(c_2\Tr_{q}^{q^n}(\gamma)^2+\Tr_{q}^{q^n}(a)\Tr_{q}^{q^n}(\gamma)^2\right)+u(c_1\Tr_{q}^{q^n}(\gamma)+\Tr_{q}^{q^n}(\gamma)&\Tr_{q}^{q^n}(a)^2+1)+c_1\Tr_{q}^{q^n}(a)+\\& c_2\Tr_{q}^{q^n}(a)^2+\Tr_{q}^{q^n}(a)^3=0.
        \end{split}
    \end{equation}
   Equation $f(X)=a$, for $a \in \F_{q^n}$, has a unique solution if and only if Equation \eqref{e313} has a unique solution in $\F_q$. If $\Tr_{q}^{q^n}(\gamma)=0$, then Equation \eqref{e313} has a unique solution, namely, 
   $$u=c_1\Tr_{q}^{q^n}(a)+c_2\Tr_{q}^{q^n}(a)^2+\Tr_{q}^{q^n}(a)^3.$$
   Hence, $f(X)$ is a permutation polynomial over $\F_{q^n}$ when 
   $\Tr_{q}^{q^n}(\gamma)=0.$
   
   We now assume that $\Tr_{q}^{q^n}(\gamma)\neq 0.$ By substituting 
   $$u=\frac{y+c_2+\Tr_{q}^{q^n}(a)}{\Tr_{q}^{q^n}(\gamma)}$$
   in Equation \eqref{e313}, we have
   \begin{equation*}
       \begin{split}
           (y^2+&(c_2+\Tr_{q}^{q^n}(a))^2)(y+c_2+\Tr_{q}^{q^n}(a))+(y^2+(c_2+\Tr_{q}^{q^n}(a))^2)(c_2+\Tr_{q}^{q^n}(a))+\\&(y+c_2+\Tr_{q}^{q^n}(a))\left(c_1+\Tr_{q}^{q^n}(a)^2+\frac{1}{\Tr_{q}^{q^n}(\gamma)}\right)+c_1\Tr_{q}^{q^n}(a)+c_2\Tr_{q}^{q^n}(a)^2+\Tr_{q}^{q^n}(a)^3=0,
   \end{split}
   \end{equation*}
   which is equivalent to
   \begin{equation}\label{e314}
       y^3+y\left(c_1+{c_2}^2+\frac{1}{\Tr_{q}^{q^n}(\gamma)}\right)+c_1c_2+\frac{c_2}{\Tr_{q}^{q^n}(\gamma)}+\frac{\Tr_{q}^{q^n}(a)}{\Tr_{q}^{q^n}(\gamma)}=0.
   \end{equation}
   Equation \eqref{e313} has a unique solution in $\F_q$ if and only if Equation \eqref{e314} has a unique solution in $\F_q.$ We now divide the proof in the following two cases.

   \textbf{Case 1.} Let $$\Tr_{q}^{q^n}(\gamma)(c_1+{c_2}^2)\neq 1.$$ Since $\Tr_{q}^{q^n}(\cdot)$ is an onto function from $\F_{q^n}$ to $\F_q$, therefore, for $c_2(c_1\Tr_{q}^{q^n}(\gamma)+1) \in \F_q$ there exists $ a \in \F_{q^n}$ such that $$\Tr_{q}^{q^n}(a)=c_2(c_1\Tr_{q}^{q^n}(\gamma)+1).$$ Now, for $a \in \F_{q^3}$ such that  $\Tr_{q}^{q^n}(a)=c_2(c_1\Tr_{q}^{q^n}(\gamma)+1)$, we have two solutions of Equation \eqref{e314}. Hence for $a \in \F_{q^3}$ with $\Tr_{q}^{q^n}(a)=c_2(c_1\Tr_{q}^{q^n}(\gamma)+1)$, $f(X)=a$ has more then one solutions. Consequently, $f(X)$ is not a permutation polynomial over $\F_{q^n}.$

   \textbf{Case 2.} If $$\Tr_{q}^{q^n}(\gamma)(c_1+{c_2}^2)=1,$$ then Equation \eqref{e314} reduces to 
   \begin{equation}\label{e315}
       y^3+c_1c_2+\frac{c_2}{\Tr_{q}^{q^n}(\gamma)}+\frac{\Tr_{q}^{q^n}(a)}{\Tr_{q}^{q^n}(\gamma)}=0.
   \end{equation}
   If $m$ is odd, then Equation \eqref{e315} has a unique solution as $\gcd(3, q-1)=1.$ If $m$ is even, then there exist $a \in \F_{q^n}$ such that Equation  \eqref{e315} has three distinct solutions as $\gcd(3, q-1)=3.$ Therefore, in this case $f(X)$ is a permutation polynomial over $\F_{q^n}$ if $m$ is odd.

   \end{proof}
   The significance of the above theorem is further illustrated by the fact that it encompasses several classes of permutation polynomials constructed in \cite{JYLQ}, as highlighted in the following remark.
\begin{rmk}
    In the case of $n=2$, Theorem \ref{T36} includes the following lemmas as special cases:
    \begin{enumerate}
        \item Lemma \ref{L26} \cite[Theorem 3.1]{JYLQ}, corresponding to $c_1=c_2=0$.
        \item Lemma \ref{L27}\cite[Theorem 3.3]{JYLQ}, corresponding to $c_1=c_2=1$.
        \item Lemma \ref{L28}\cite[Theorem 3.5]{JYLQ}, corresponding to $c_1=1, c_2=0$.
        \item Lemma \ref{L29}\cite[Theorem 3.6]{JYLQ}, corresponding to $c_1=0, c_2=1$.
    \end{enumerate}
\end{rmk}
Next, we determine the compositional inverse of
\[
f(X)=X+\gamma \Tr_{q}^{q^n}\left(c_1X+c_2X^2+X^2\Tr_{q}^{q^n}(X)\right)\in \mathbb{F}_{q^n}[X],
\]
where \(c_1,c_2\in\mathbb{F}_q\) are as defined in Theorem \ref{T36}. Before providing the compositional inverse, we first establish the following lemmas.

\begin{lem}\label{L41}
    Let $q=2^m$ and $f(X)=X+\gamma \Tr_{q}^{q^n}(c_1X+c_2X^2+X^2\Tr_{q}^{q^n}(X))\in \F_{q^n}[X]$, where $c_1, c_2 \in \F_q$ such that  $\Tr_{q}^{q^n}(\gamma)(c_1+c_2^2)=1$. Then $$\Tr_{q}^{q^n}((1+\gamma)X) \circ f(X)=\Tr_{q}^{q^n}((1+\gamma)X)+\frac{c_1+c_2^2+1}{(c_1+c_2^2)^2}\left(\frac{f(X)+X}{\gamma}\right).$$
\end{lem}
\begin{proof}
    We have 
    \begin{equation*}
    \begin{split}
        \Tr_{q}^{q^n}((1+\gamma)X) \circ f(X)=&\Tr_{q}^{q^n}((1+\gamma)(X+\gamma\Tr_{q}^{q^n}(c_1X+c_2X^2+X^2\Tr_{q}^{q^n}(X))))
        \\=& \Tr_{q}^{q^n}((1+\gamma)X)+\Tr_{q}^{q^n}(c_1X+c_2X^2+X^2\Tr_{q}^{q^n}(X))\Tr_{q}^{q^n}(\gamma+\gamma^2)
        \\=&\Tr_{q}^{q^n}((1+\gamma)X)+\Tr_{q}^{q^n}(c_1X+c_2X^2+X^2\Tr_{q}^{q^n}(X))\left(\frac{1}{c_1+c_2^2}+\frac{1}{(c_1+c_2^2)^2}\right)
        \\=&\Tr_{q}^{q^n}((1+\gamma)X)+\frac{c_1+c_2^2+1}{(c_1+c_2^2)^2}\left(\frac{f(X)+X}{\gamma}\right).
    \end{split}
         \end{equation*}
\end{proof}

\begin{lem}\label{L42}
     Let $q=2^m$ and $f(X)=X+\gamma \Tr_{q}^{q^n}(c_1X+c_2X^2+X^2\Tr_{q}^{q^n}(X))\in \F_{q^n}[X]$, where $c_1, c_2 \in \F_q$ such that  $\Tr_{q}^{q^n}(\gamma)(c_1+c_2^2)=1$. Then
     $$\left((c_1+c_2^2)\Tr_{q}^{q^n}(X)+c_2^3\right) \circ f(X)=(\Tr_{q}^{q^n}(X)+c_2)^3.$$
\end{lem}
\begin{proof}
    Here,
    \begin{equation*}
        \begin{split}
            \left((c_1+c_2^2)\Tr_{q}^{q^n}(X)+c_2^3\right) \circ f(X)=&(c_1+c_2^2)\Tr_{q}^{q^n}(f(X))+c_2^3
            \\=&(c_1+c_2^2)\Tr_{q}^{q^n}(X+\gamma \Tr_{q}^{q^n}(c_1X+c_2X^2+X^2\Tr_{q}^{q^n}(X)))+c_2^3
            \\=&(c_1+c_2^2)\left(\Tr_{q}^{q^n}(X)+\frac{\Tr_{q}^{q^n}(c_1X+c_2X^2+X^2\Tr_{q}^{q^n}(X))}{c_1+c_2^2}\right)+c_2^3
            \\=&c_2^2\Tr_{q}^{q^n}(X)+c_2\Tr_{q}^{q^n}(X)^2+\Tr_{q}^{q^n}(X)^3+c_2^3
            \\=&(\Tr_{q}^{q^n}(X)+c_2)^3.
        \end{split}
    \end{equation*}
\end{proof}
\begin{lem}\label{L43}
     Let $q=2^m$ and $f(X)=X+\gamma \Tr_{q}^{q^n}(c_1X+c_2X^2+X^2\Tr_{q}^{q^n}(X))\in \F_{q^n}[X]$, where $c_1, c_2 \in \F_q$. Then
     $$\left(\displaystyle\sum_{\substack{i=1\\ i\neq j}}^{n-1}\sum_{j=1}^{n-1}\gamma^{q^i}X^{q^j}\right) \circ f(X)=\displaystyle\sum_{\substack{i=1\\ i\neq j}}^{n-1}\sum_{j=1}^{n-1}\gamma^{q^i}X^{q^j}.$$
\end{lem}
\begin{proof}
    We can see 
    \begin{equation*}
        \begin{split}
            \left(\displaystyle\sum_{\substack{i=1\\ i\neq j}}^{n-1}\sum_{j=1}^{n-1}\gamma^{q^i}X^{q^j}\right) \circ f(X)=&\displaystyle\sum_{\substack{i=1\\ i\neq j}}^{n-1}\sum_{j=1}^{n-1}\gamma^{q^i}f(X)^{q^j}
            \\=&\displaystyle\sum_{\substack{i=1\\ i\neq j}}^{n-1}\sum_{j=1}^{n-1}\gamma^{q^i}(X+\gamma \Tr_{q}^{q^n}(c_1X+c_2X^2+X^2\Tr_{q}^{q^n}(X)))^{q^j}
            \\=&\displaystyle\sum_{\substack{i=1\\ i\neq j}}^{n-1}\sum_{j=1}^{n-1}\gamma^{q^i}\left(X^{q^j}+\Tr_{q}^{q^n}(c_1X+c_2X^2+X^2\Tr_{q}^{q^n}(X))\gamma^{q^j}\right)
            \\=&\displaystyle\sum_{\substack{i=1\\ i\neq j}}^{n-1}\sum_{j=1}^{n-1}\gamma^{q^i}X^{q^j}+\Tr_{q}^{q^n}(c_1X+c_2X^2+X^2\Tr_{q}^{q^n}(X))\displaystyle\sum_{\substack{i=1\\ i\neq j}}^{n-1}\sum_{j=1}^{n-1}\gamma^{q^i+q^j}
            \\=&\displaystyle\sum_{\substack{i=1\\ i\neq j}}^{n-1}\sum_{j=1}^{n-1}\gamma^{q^i}X^{q^j}
        \end{split}
    \end{equation*}
    as $\displaystyle\sum_{\substack{i=1\\ i\neq j}}^{n-1}\sum_{j=1}^{n-1}\gamma^{q^i+q^j}=0.$
\end{proof}
In the following proposition, we give the compositional inverse of $f(X)$ defined in Theorem \ref{T36}.
\begin{prop}\label{P41}
     Let $q=2^m$, where $m$ is a positive integer. Let $f(X)=X+\gamma \Tr_{q}^{q^n}(c_1X+c_2X^2+X^2\Tr_{q}^{q^n}(X))\in \F_{q^n}[X]$, where $c_1, c_2 \in \F_q$. Then
     \[
     f^{-1}(X)=
     \begin{cases}
         f(X) & \text{if } \Tr_{q}^{q^n}(\gamma)=0,\\
         g(X) & \text{if } \Tr_{q}^{q^n}(\gamma)(c_1+c_2^2)=1  \text{and $m$ is odd}, 
     \end{cases}
     \]
     where 
     \begin{equation*}
     \begin{split}
         g(X)=&\Tr_{q}^{q^n}((1+\gamma)X)+\gamma\left((c_1+c_2^2)\Tr_{q}^{q^n}(X)+c_2^3\right)^t+\gamma c_2+\displaystyle\sum_{\substack{i=1\\ i\neq j}}^{n-1}\sum_{j=1}^{n-1}\gamma^{q^i}X^{q^j}+\\
     &\sum_{j=1}^{n-1}\left(\frac{1}{c_1+c_2^2}+1\right)X^{q^j}+\left(\left(\frac{1}{c_1+c_2^2}+1\right)\sum_{j=1}^{n-1}\gamma^{q^j}+\frac{c_1+c_2^2+1}{(c_1+c_2^2)^2}\right)\\
     &\left((c_1+c_2^2)\left((c_1+c_2^2)\Tr_{q}^{q^n}(X)+c_2^3\right)^t+(c_1+c_2^2)\Tr_{q}^{q^n}(X)+c_2^3+c_1c_2\right)
     \end{split}
     \end{equation*}
     such that $3t \equiv 1 \mod(2^m-1).$
      
\end{prop}
\begin{proof}
    We first assume that $\Tr_{q}^{q^n}(\gamma)=0.$ Let $\alpha=\gamma \Tr_{q}^{q^n}(c_1X+c_2X^2+X^2\Tr_{q}^{q^n}(X))$ and therefore
  \begin{equation*}
 \begin{split}
      f(f(X))&=X+\alpha+\gamma\Tr_{q}^{q^n}(c_1(X+\alpha)+c_2(X+\alpha)^2+(X+\alpha)^2\Tr_{q}^{q^n}(X+\alpha))\\
      &=X+\alpha+\gamma \Tr_{q}^{q^n}(c_1X+c_1\alpha +c_2X^2+c_2\alpha^2+(X^2+\alpha^2)( \Tr_{q}^{q^n}(X)+ \Tr_{q}^{q^n}(\alpha)))
      \\&=X+\alpha+\gamma c_1 \Tr_{q}^{q^n}(X)+\gamma c_2 \Tr_{q}^{q^n}(X^2)+\gamma \Tr_{q}^{q^n}(X^2)\Tr_{q}^{q^n}(X)
      \end{split}
  \end{equation*}
as $\Tr_{q}^{q^n}(\alpha)=0$. Hence, $f(f(X))=X.$

  Next, we assume that $\Tr_{q}^{q^n}(\gamma)(c_1+c_2^2)=1$ and $m$ is odd. We have
  \begin{equation}\label{e410}
      g(f(X))=A+B+C+D,
  \end{equation} where
  \begin{equation*}
      \begin{split}
          A=&\Tr_{q}^{q^n}((1+\gamma)f(X)),
          \\B=&\gamma\left((c_1+c_2^2)\Tr_{q}^{q^n}(f(X))+c_2^3\right)^t+\gamma c_2,
          \\C=&\displaystyle\sum_{\substack{i=1\\ i\neq j}}^{n-1}\sum_{j=1}^{n-1}\gamma^{q^i}f(X)^{q^j}+\sum_{j=1}^{n-1}\left(\frac{1}{c_1+c_2^2}+1\right)f(X)^{q^j}, \text{ and }
          \\D=&\left(\left(\frac{1}{c_1+c_2^2}+1\right)\sum_{j=1}^{n-1}\gamma^{q^j}+\frac{c_1+c_2^2+1}{(c_1+c_2^2)^2}\right)\left((c_1+c_2^2)\left((c_1+c_2^2)\Tr_{q}^{q^n}(f(X))+c_2^3\right)^t+\right.\\&\left.(c_1+c_2^2)\Tr_{q}^{q^n}(f(X))+c_2^3+c_1c_2\right).
      \end{split}
  \end{equation*}
  Using Lemma \ref{L41}, Lemma \ref{L42} and Lemma \ref{L43}, we get
  \begin{equation*}
      \begin{split}
          A=&\Tr_{q}^{q^n}((1+\gamma)X)+\frac{c_1+c_2^2+1}{(c_1+c_2^2)^2}\left(\Tr_{q}^{q^n}(c_1X+c_2X^2+X^2\Tr_{q}^{q^n}(X))\right),
          \\B=& \gamma \Tr_{q}^{q^n}(X),
          \\C=&\displaystyle\sum_{\substack{i=1\\ i\neq j}}^{n-1}\sum_{j=1}^{n-1}\gamma^{q^i}X^{q^j}+\sum_{j=1}^{n-1}\left(\frac{1}{c_1+c_2^2}+1\right)(X^{q^j}+\gamma^{q^j}\Tr_{q}^{q^n}(c_1X+c_2X^2+X^2\Tr_{q}^{q^n}(X)),  \text{ and }
          \\D=&\left(\left(\frac{1}{c_1+c_2^2}+1\right)\sum_{j=1}^{n-1}\gamma^{q^j}+\frac{c_1+c_2^2+1}{(c_1+c_2^2)^2}\right)\Tr_{q}^{q^n}(c_1X+c_2X^2+X^2\Tr_{q}^{q^n}(X)).
      \end{split}
  \end{equation*}
  We now put the values of $A$, $B$, $C$ and $D$ in Equation \ref{e410}. Consequently, we obtain
  \begin{equation*}
      \begin{split}
          g(f(X))=&\Tr_{q}^{q^n}((1+\gamma)X)+\gamma \Tr_{q}^{q^n}(X)+\displaystyle\sum_{\substack{i=1\\ i\neq j}}^{n-1}\sum_{j=1}^{n-1}\gamma^{q^i}X^{q^j}+\sum_{j=1}^{n-1}\left(\frac{1}{c_1+c_2^2}+1\right)X^{q^j}
          \\=&\Tr_{q}^{q^n}(X)+\Tr_{q}^{q^n}(\gamma X)+\gamma \Tr_{q}^{q^n}(X)+\displaystyle\sum_{\substack{i=1\\ i\neq j}}^{n-1}\sum_{j=1}^{n-1}\gamma^{q^i}X^{q^j}+\sum_{j=1}^{n-1}\left(\frac{1}{c_1+c_2^2}+1\right)X^{q^j}
          \\=&X+\sum_{j=1}^{n-1}(1+\gamma+\gamma^{q^j})X^{q^j}+\displaystyle\sum_{\substack{i=1\\ i\neq j}}^{n-1}\sum_{j=1}^{n-1}\gamma^{q^i}X^{q^j}+\sum_{j=1}^{n-1}\left(\frac{1}{c_1+c_2^2}+1\right)X^{q^j}
          \\=&X+\sum_{j=1}^{n-1}\left(\gamma+\gamma^{q^j}+\frac{1}{c_1+c_2^2}\right)X^{q^j}+\displaystyle\sum_{\substack{i=1\\ i\neq j}}^{n-1}\sum_{j=1}^{n-1}\gamma^{q^i}X^{q^j}
          \\=&X
      \end{split}
  \end{equation*}
  as $\Tr_{q}^{q^n}(\gamma)=\frac{1}{c_1+c_2^2}.$ Hence, $g(X)$ is the compositional inverse of $f(X).$
  \end{proof}
  \begin{rmk}
      For $n=2$, Proposition \ref{P41} yields Theorems 3.1--3.4 of \cite{SKP} as special cases.
  \end{rmk}

\end{document}